\documentclass[a4paper]{amsart}

\usepackage{amsmath,amssymb,amsthm,graphicx,enumerate}

\newtheorem{theorem}{Theorem}

\theoremstyle{definition}
\newtheorem{fact}{Fact}
\newtheorem{example}{Example}

\begin{document}

\title{An improved lower bound for finite additive 2-bases}

\author{Jukka Kohonen}
\address{Department of Computer Science, Aalto University, Espoo, Finland}
\email{jukka.kohonen@aalto.fi}

\begin{abstract}
  A set of non-negative integers $A$ is an additive 2-basis with range
  $n$, if its sumset $A+A$ contains $0, 1, \ldots, n$ but not $n+1$.
  Explicit bases are known with arbitrarily large size $|A|=k$ and
  $n/k^2 \ge 2/7 > 0.2857$.  We present a more general construction
  and improve the lower bound to $85/294 > 0.2891$.
\end{abstract}

\maketitle


\section{Introduction}
\label{sec:intro}
A set of non-negative integers $A$ is an additive 2-basis of size
$k=|A|$ and range $n=n(A)$, if its sumset $A+A$ contains the integers
$0, 1, \ldots, n$ but not $n+1$.  The maximal ranges \[n(k) =
\max_{|A|=k} n(A)\] are known up to $n(25)=212$ \cite{kohonen2014}.
Lacking an explicit formula for $n(k)$, attention has been paid to
upper and lower bounds proportional to $k^2$.  An easy counting
argument shows that $n(k) \le k^2/2+k/2$.  The simple construction $A
= \{0, 1, \ldots, t, 2t, \ldots, t^2\}$ shows that $n(k) \ge k^2/4$.

The upper bound has been improved several times.  Yu \cite{yu2015}
recently proved that
\[
\limsup_{k \to \infty} \frac{n(k)}{k^2} \le 0.4585.
\]
For the lower bound, an explicit construction by Mrose \cite{mrose1979}
shows that
\[
\liminf_{k \to \infty} \frac{n(k)}{k^2} \ge 2/7 > 0.2857.
\]
Kl\o ve and Mossige \cite{mossige1981} presented another construction
that achieves the same factor $2/7$.  In this note we show that
\begin{equation}
  \liminf_{k \to \infty} \frac{n(k)}{k^2} \ge 85/294 > 0.2891.
  \label{eq:liminf}
\end{equation}

For simplicity we define the size of a basis as $k=|A|$, including the
necessary zero element.  Often in the literature the zero is not counted,
but this makes no difference in the asymptotic ratios.

\section{Generalized Mrose basis}
\label{sec:const}

For finite arithmetic progressions, translation of a set by a
constant, and pointwise multiplication we use the notation
\begin{align*}
  [a, b] &= \{a, a+1, \ldots, b\},\\
  [a, (m), b] &= \{a, a+m, \ldots, b\},\\
  A + h &= \{a+h : a \in A\},\\
  h \cdot A &= \{ha : a \in A\}.
\end{align*}

Let an integer $t \ge 2$ be given.  We will build an additive basis
from translated copies of three \emph{elementary segments}:
\begin{align*}
  V &= [0, t], \\
  H &= [0, (t), t^2-t], \\
  S &= [0, (t+1), t^2-1].
\end{align*}
Note that $|V|=t+1$ and $|H|=|S|=t$.  It will be beneficial to
visualize integers as mapped to planar coordinates by $i \mapsto
(\lfloor i/t \rfloor, i \bmod t)$.  Then $V$ is a vertical line with
an extra element, $H$ is a horizontal line, and $S$ is a slanted line,
as illustrated in Figure~\ref{fig:segments}.

\begin{figure}[tb]
  \includegraphics[height=2.5cm]{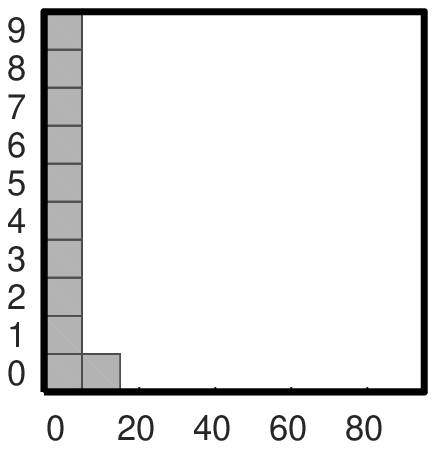}
  \hfill
  \includegraphics[height=2.5cm]{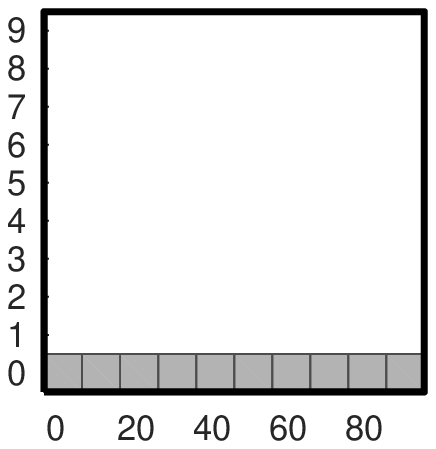}
  \hfill
  \includegraphics[height=2.5cm]{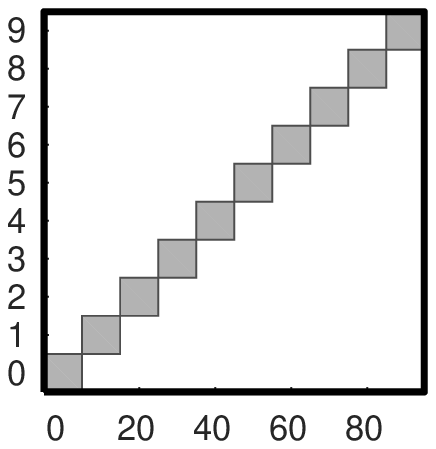}
  \caption{Three elementary segments: $V$, $H$, and $S$ (with $t=10$).}
  \label{fig:segments}
\end{figure}

\begin{figure}[tb]
  \includegraphics[height=2.5cm]{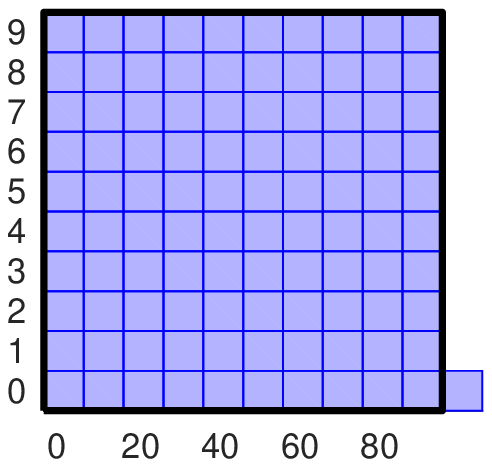}
  \hfill
  \includegraphics[height=2.5cm]{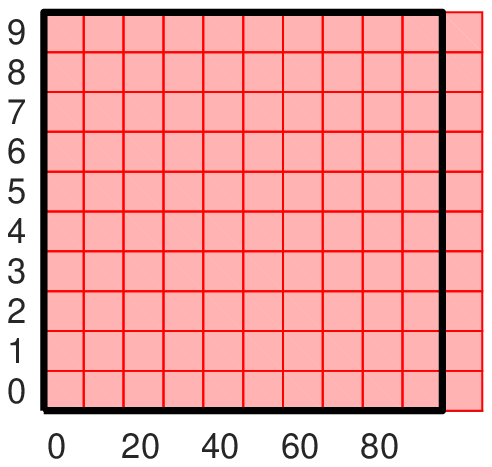}
  \hfill
  \includegraphics[height=2.5cm]{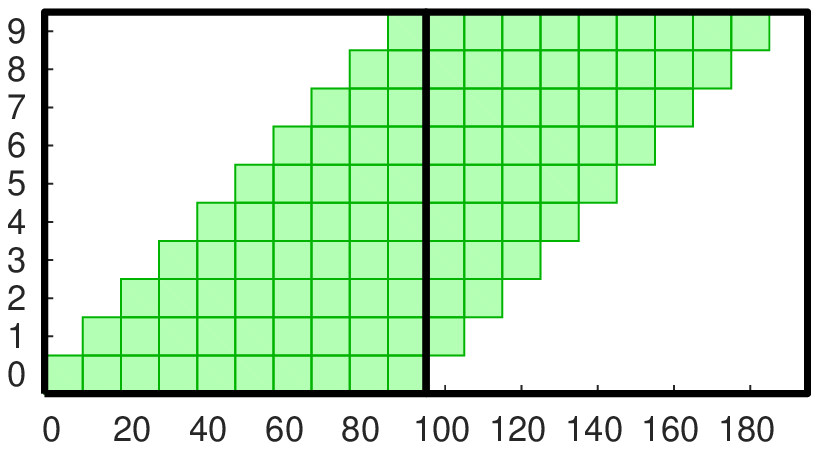}
  \caption{Sumsets of elementary segments: $V\!+\!H$ (blue), $V\!+\!S$
    (red), and $H\!+\!S$ (green).}
  \label{fig:tiles}
\end{figure}

The elementary segments give rise to six pairwise sumsets, but
$V\!+\!V$, $H\!+\!H$ and $S\!+\!S$ are of negligible size $O(t)$ and
will be ignored here.  We concentrate on $V\!+\!H$, $V\!+\!S$, and the
parallelogram-shaped $P=H\!+\!S$, each of which has size at least
$t^2$ as illustrated in Figure~2.  The following facts are easily
verified.
\begin{fact}
  Both $V+H$ and $V+S$ contain the square $Q = [0, t^2-1]$.
\end{fact}
\begin{fact}
  The union of two consecutive parallelograms $P$ and $P+t^2$ contains
  the square $Q+t^2$.
  \label{fact:parallelogram}
\end{fact}
\begin{fact}
If elementary segments are translated, their sumsets are likewise
translated: for example, $(V+i) + (H+j) = (V\!+\!H)+(i+j) \supseteq
Q+(i+j)$.
\end{fact}
Consider a basis constructed from $\ell$ elementary segments, placed
at specified multiples of $t^2$.  More precisely, if $I,J,K$ are sets
of non-negative integers, with $|I|+|J|+|K|=\ell$, we define
\begin{equation}
  A = (V + t^2\!\cdot\!I) \;\cup\; (H + t^2\!\cdot\!J) \;\cup\; (S + t^2\!\cdot\!K).
  \label{eq:vhs}
\end{equation}
We say that $A$ is a \emph{generalized Mrose basis} with placement
$(I,J,K)$ and segment length $t$.  Using the aforementioned facts we
have
\begin{align*}
  & A+A \\
  & \supseteq
  \bigl( (V\!+\!H) + (I'\!+\!J') \bigr) \;\cup\;
  \bigl( (V\!+\!S) + (I'\!+\!K') \bigr) \;\cup\;
  \bigl( (H\!+\!S) + (J'\!+\!K') \bigr) \\
  & \supseteq
  \bigl( Q + (I'\!+\!J') \bigr) \;\cup\;
  \bigl( Q + (I'\!+\!K') \bigr) \;\cup\;
  \bigl( P + (J'\!+\!K') \bigr),
\end{align*}
where we have written $I'=t^2\!\cdot\!I$ for brevity, and $J'$, $K'$
likewise.  In other words, $A+A$ covers squares at locations
$t^2\!\cdot\!(I\!+\!J)$ and $t^2\!\cdot\!(I\!+\!K)$, and parallelograms
at locations $t^2\!\cdot\!(J\!+\!K)$.

We now face the combinatorial problem of choosing and placing $\ell$
copies of elementary segments, so as to maximize the number $m$ of
covered consecutive squares beginning from~$0$.  If $m$ such squares
are covered, then \eqref{eq:vhs} is an additive basis of size $k \le
\ell (t+1)$ and range
\[
n \ge mt^2-1 \ge ck^2r-1,
\]
where $c = m/\ell^2$ and $r=(t/(t+1))^2$.  The factor $r$ appears
because vertical segments have $t+1$ elements, but $r$ tends to 1 as
$t\to\infty$.

\begin{figure}[tb]
  \includegraphics[width=\textwidth]{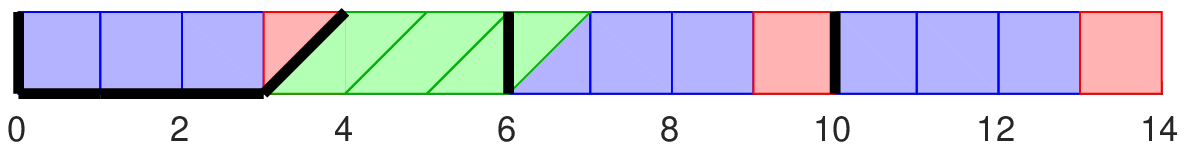}

  \vspace{2cm}
  
  \includegraphics[width=\textwidth]{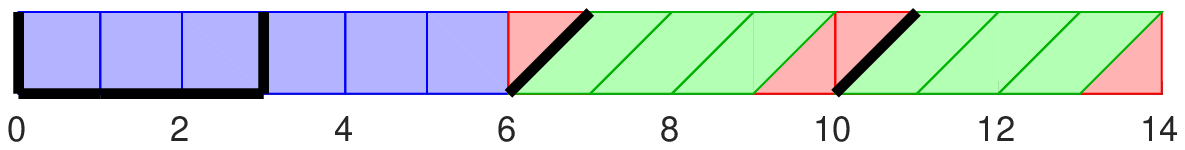}
  \caption{Two constructions with $\ell=7$, $m=14$, and $t$ large.  A
    unit square represents an interval of $t^2$ integers.  The
    elementary segments are shown as thick black lines.  Copies of
    $V\!+\!H$ and $V\!+\!S$ are shown as blue and red squares,
    respectively.  Copies of $H\!+\!S$ are shown as green
    parallelograms.  Some squares are only partially visible due to
    overlap.}
  \label{fig:mroseklove}
\end{figure}

\begin{example}
  Choosing $\ell=2$ and $I=J=\{0\}$, $K=\varnothing$ we have
  $I\!+\!J=\{0\}$, thus $m=1$ and $c=m/\ell^2=1/4$.  This is
  essentially the simple construction mentioned in the introduction.
\end{example}

\begin{example}
  Choosing $\ell=7$ and $I=\{0, 6, 10\}$, $J=\{0,1,2\}$, $K=\{3\}$
  gives a basis that is structurally similar to that of Mrose
  \cite{mrose1979}, illustrated in the top of
  Figure~\ref{fig:mroseklove}.  The copies of $V\!+\!H$ and $V\!+\!S$
  cover squares $Q + t^2 \cdot \{0,1,2,3, 6,7,\ldots,13\}$.  The
  copies of $H\!+\!S$ cover parallelo\-grams $P+t^2\cdot\{3,4,5\}$,
  containing in particular the squares $Q+t^2\cdot\{4,5\}$.  Since
  $m=14$ consecutive squares from $0$ are covered, we have
  $c=m/\ell^2=2/7$, asymptotically matching Mrose's result.
\end{example}

\begin{example}
  Choosing $\ell=7$ and $I=\{0, 3\}$, $J=\{0,1,2\}$, $K=\{6,10\}$ we
  obtain another basis with $m=14$, illustrated in the bottom of
  Figure~\ref{fig:mroseklove}.  This basis is similar to that of Kl\o
  ve and Mossige \cite{mossige1981}.
\end{example}

One can now try different sizes $\ell$ and placements $(I,J,K)$,
seeking to maximize $m/\ell^2$.  With a simple computer program we
searched through placements of size $\ell \le 17$, but found none with
$m/\ell^2>2/7$.  However, from a combination of computer-based search
and manual design, we have the following result.

\begin{figure}[tb]
  \includegraphics[width=\textwidth]{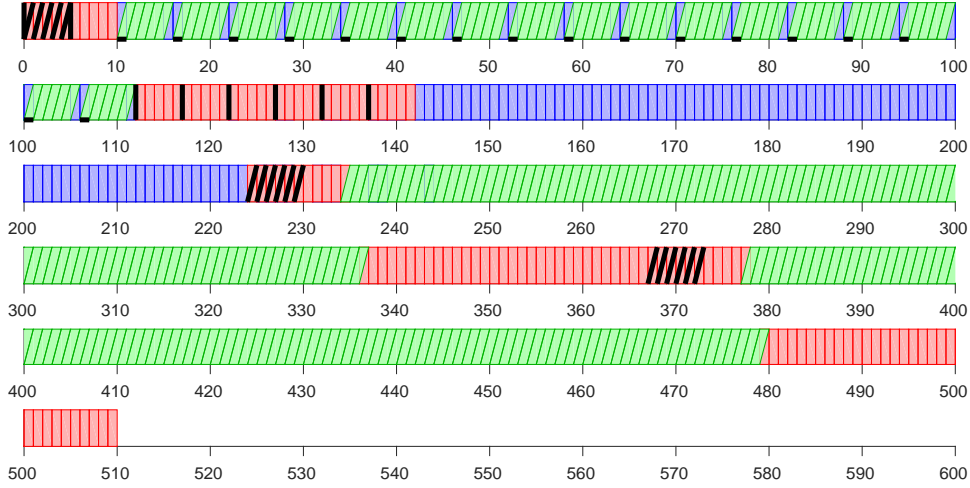}
  \caption{A construction of $\ell=42$ elementary segments. shown as
    thick black lines.  Copies of $V+H$ are shown in blue, copies of
    $V+S$ in red, and copies of $H+S$ in green (some not visible due
    to overlap).}
  \label{fig:basis42}
\end{figure}

\begin{theorem}
  There is a placement $(I,J,K)$ with $\ell=42$ such that $A+A$ covers
  $m=510$ consecutive squares beginning from zero.
\end{theorem}

\begin{proof}
  Let
  \begin{align*}
    I &= \{0,5\} \cup [112, (5), 137], \\
    J &= [10, (6), 106], \\
    K &= [0,4] \cup [224,229] \cup [367,372].
  \end{align*}
  Note that $|I|+|J|+|K|=8+17+17=42$.  Let us verify that $2A=A+A$
  covers the squares $Q+t^2\cdot[0,509]$ as claimed.  The proof
  proceeds by subintervals and is illustrated in
  Figure~\ref{fig:basis42}.

  \begin{enumerate}[(i)]
    \item
      $Q+t^2\cdot[0,9] \subseteq 2A$, since $[0,9] \subseteq
      \{0,5\}+[0,4] \subseteq I+K$.
  
    \item
      For each $j \in J$, we observe that $j+[0,4] \subseteq J+K$, so
      $2A$ covers consecutive parallelograms $P + t^2\cdot(j+[0,4])$,
      and in particular the squares $Q + t^2\cdot(j+[1,4])$.
      Combining this with the fact that the squares $Q +
      t^2\cdot(j+\{0,5\})$ are covered, it follows that $Q +
      t^2\cdot(j+[0,5])$ is covered.  As this holds for all $j \in J$,
      we see that $Q+t^2\cdot[10,111] \subseteq 2A$.

    \item
      $Q+t^2\cdot[112,141] \subseteq 2A$, since $[112,141] \subseteq [112, (5), 137] + [0,4] \subseteq I+K$.

    \item
  $Q+t^2\cdot[142,223] \subseteq 2A$, since $[142,223] \subseteq [112,
      (5), 137] + [10, (6), 106] \subseteq I+J$.

    \item  
      $Q+t^2\cdot[224,234] \subseteq 2A$, since $[224,234] \subseteq
      \{0,5\} + [224,229] \subseteq I+K$.

    \item
      $Q+t^2\cdot[235,335] \subseteq P+t^2\cdot[234,335]$ by
      Fact~\ref{fact:parallelogram}.  These consecutive parallelograms
      are covered by $2A$ since $[234,335] \subseteq
      [10,(6),106]+[224,229] \subseteq J+K$.
  
    \item
      $Q+t^2\cdot[336,366] \subseteq 2A$, since $[336,366] \subseteq
      [112, (5), 137] + [224,229] \subseteq I+K$.
  
    \item
      $Q+t^2\cdot[367,377] \subseteq 2A$, since $[367,377] \subseteq
      \{0,5\} + [367,372] \subseteq I+K$.
  
    \item
      $Q+t^2\cdot[378,478] \subseteq P+t^2\cdot[377,478]$ by
      Fact~\ref{fact:parallelogram}.  These consecutive parallelograms
      are covered by $2A$ since $[377,478] \subseteq
      [10,(6),106]+[367,372] \subseteq J+K$.
  
    \item
      $Q+t^2\cdot[479,509] \subseteq 2A$, since $[479,509] \subseteq
      [112, (5), 137] + [367,372] \subseteq I+K$.
  \end{enumerate}
\end{proof}

With the placement described above we have
\[
c=m/\ell^2 = 510/42^2 = 85/294.
\]
In more detail, for any integer $t \ge 2$, this placement gives a
generalized Mrose basis of size $k=42t+7$ and range $n\ge 510t^2$.  We
thus have
\[
\lim_{t\to\infty} \frac{n}{k^2} \ge 85/294 > 0.2891,
\]
justifying claim \eqref{eq:liminf}.

\section{Discussion of further improvement}

Striving for simplicity, we have opted to place the copies of
elementary segments at integer multiples of $t^2$.  It would be
possible to move them slightly further: for example, in Mrose's
original construction the first segment is $[0,t]$, and the second
segment begins at $2t$.  However, such changes would in general only
extend the range by an amount linear in $t$, and thus would not
improve the asymptotic ratio $n/k^2$.

For any additive basis of the form \eqref{eq:vhs}, a counting argument
provides an upper bound on $|A+A|$ (and hence on $n(A)$).  Let the
numbers of elementary segments of each kind be $\ell_I=|I|$,
$\ell_J=|J|$ and $\ell_K=|K|$.  Observing that $V+V$, $H+H$ and $S+S$
have size $O(t)$, and the three ``useful'' sumsets $V+H$, $V+S$ and
$H+S$ have size $t^2 + O(t)$, we have
\[
|A+A| \le (\ell_I\ell_J+\ell_I\ell_K+\ell_J\ell_K)t^2 + O(t).
\]
Subject to the constraint $\ell_I+\ell_J+\ell_K=\ell$, we have
\[
|A+A| \le (1/3)\ell^2t^2 + O(t) = (1/3)k^2 + O(k).
\]
In other words, no matter how well the placement $I,J,K$ is chosen,
the asymptotic ratio $n/k^2$ achievable through this construction from
three elementary segments cannot essentially exceed $1/3$.  If one
aims to exceed this ratio, one may want to consider four or more kinds
of elementary segments.  The challenge is then twofold: first, to
design elementary segments with conveniently-shaped sumsets that fit
together well; and second, to find a good placement of their copies.
This approach could be seen as a decomposition of an additive basis
into a ``structured part'' (elementary segments with fixed structure,
but arbitrary size) and an ``unstructured part'' (placement of
segments, perhaps through random or exhaustive search), reminiscent of
Bibak's general suggestion~\cite[p.~114]{bibak2013}.

\section*{Acknowledgements}
The research leading to these results has received funding from the
European Research Council under the European Union's Seventh Framework
Programme (FP/2007-2013) / ERC Grant Agreement 338077 ``Theory and
Practice of Advanced Search and Enumeration''.

\section*{References}

\bibliographystyle{plain} 
\bibliography{refs}

\end{document}